\documentclass[reqno]{amsart}
\usepackage{amscd}
\usepackage[dvips]{graphics}

\def\beq{\begin{equation}}
\def\eeq{\end{equation}}
\def\ba{\begin{array}}
\def\ea{\end{array}}

\numberwithin{equation}{section}
\newenvironment{abs}{\textbf{Abstract}\mbox{  }}{ }
\newenvironment{key words}{\textbf{Keywords}\mbox{  }}{ }
\newtheorem{theorem}{Theorem}[section]

\newtheorem{corollary}[theorem]{\textbf{Corollary}}
\newtheorem{proposition}[theorem]{\textbf{Proposition}}
\newtheorem{lemma}[theorem]{Lemma}
\renewenvironment{proof}{\noindent{\textbf{Proof.}}}{\hfill$\Box$}
\theoremstyle{remark}

\theoremstyle{plain}

\begin{document}
\title{\textbf{Nonlinear elliptic equations on the upper half space}}
\author  {Sufang Tang, Lei Wang  and Meijun Zhu}

\address{Sufang Tang, School of Statistics, Xi'an University of Finance and
Economics, Xi'an, Shaanxi, 710100, P. R. China}
\email {sufangtang@163.com}

\address{ Lei Wang, Academy of Mathematics and Systems Sciences, Chinese Academy of Sciences, Beijing 100190, P.R. China, and Department of Mathematics, The University of Oklahoma, Norman, OK 73019, USA}

\email{wanglei@amss.ac.cn}

\address{ Meijun Zhu, Department of Mathematics,
The University of Oklahoma, Norman, OK 73019, USA}

\email{mzhu@math.ou.edu}

\maketitle

\noindent
\begin{abs}
In this paper we shall classify  all positive
solutions of $ \Delta u =a u^p$ on the upper half
space $ H =\Bbb{R}_+^n$
with nonlinear
boundary condition $ {\partial u}/{\partial t}= - b u^q  $ on $\partial H$ for both positive parameters  $a, \ b>0$. We will prove that for $p \ge {(n+2)}/{(n-2)}, 1\leq q<{n}/{(n-2)}$ (and $n \ge 3$)  all positive solutions are functions of last variable;  for $p= {(n+2)}/{(n-2)}, q= {n}/{(n-2)}$ (and $n \ge 3$)  positive
solutions must be either some functions depending only on last  variable, or  radially symmetric functions.
\end{abs}\\

\begin{key words} Curvature equation, Kelvin transform, Moving plane method,  Moving sphere method
\end{key words}\\
\section{\textbf{Introduction}\label{Section 1}}

Let $ H =\Bbb{R}_+^n= \{ (x', t) \ | \ x'=(x_1, x_2,..., x_{n-1}) \in
\Bbb{R}^{n-1}, \ t>0 \}$ be the upper half space in $\Bbb{R}^n$ with $n
\ge 3$.  We are interested in the following elliptic equation with a nonlinear boundary condition.
\begin{equation}\label{elliptic}
\left\{
\aligned
& \Delta u  = au^p, \ \ u \ge 0 \ \ {in} \ \ \ H  =\Bbb{R}_+^n,\\
& \frac {\partial u}{\partial t} = bu^{q} \ \ \  {on} \ \ \
\partial H= \Bbb{R}^{n-1}.
\endaligned
\right.
\end{equation}

Equation \eqref{elliptic} was early studied by J. Escobar \cite {Escobar90} in his work on Yamabe problem on compact manifolds with boundary. Let $(M, g_0)$ be a compact Riemannian manifold with boudary of dimension $n\geq 3$. The Yamabe problem is to find metrics $g$ conformally equivalent to $g_0$ of constant scalar curvature $R_g$, with constant mean curvature $h_g$ on $\partial M$. When $M$ is the unit ball in $\mathbb{R}^n$, since $M$ is conformall equivalent to $\mathbb{R}^n_+$, the problem is reduced to equation \eqref{elliptic} for the critical case ($p=\frac{n+2}{n-2}$ and $q=\frac{n}{n-2}$). With suitable decay assumption, Escobar was able to obtain the classification result for $a<0, b<0,  p=\frac{n+2}{n-2}$ and $q=\frac{n}{n-2}$ (which mean $R_g\geq0$ and $h_g<0$ for the manifold). Later, the decay condition was removed by Li and Zhu \cite{LZ1995},  and  many interesting results on the classification of positive solutions of equation \eqref{elliptic} for general cases appeared.  We summarize the known results below.

(i) For  $a=0, b<0,$  Hu \cite{Hu1994} showed that there is no positive classical solution for  subcritical case ($1\leq q<\frac{n}{n-2}$ and $n\geq 3$); For critical case $q=\frac{n}{n-2}$, Ou \cite{Ou1996} proved that all  positive classical solutions are fundamental solutions of the Laplace equation multiplied by proper constants.   The classical moving plane method (see, for example, \cite{GNN1981})  was used in these papers.

(ii) For  $a<0, b\in \mathbb{R},$  by introducing the method of moving spheres, Li and Zhu \cite{LZ1995} classified all  nonnegative classical solutions without the decay assumption at infinity for the critical case $p=\frac{n+2}{n-2}, q=\frac{n}{n-2}$ (and $n\geq 3$). See, also the paper by Chipot, Shafrir and Fila \cite{CSF1996} 


(iii) For  $a \ge 0,  b>0,$ Lou and Zhu \cite{LZ1999} classified all nonnegative classical solutions for  $q>1$ and $n\geq 2$ if $a=0, \ b>0$; and  for $p, \ q>1$ and $n \ge 2$  if $a, \ b>0$ by using the method of moving planes and the technics of lifting dimensions.

(iv) For $a<0, b<0,$ Li and Zhang \cite{LZ2003} proved that there is no positive classical solution for $0\leq p\leq\frac{n+2}{n-2}, -\infty< q\leq\frac{n}{n-2}, p+q< (\frac{n+2}{n-2})+(\frac{n}{n-2})$ (and $n\geq 3$) by using the
method of moving spheres.

\medskip

In this paper, we shall study the remaining case: $a>0, \ b<0$. After rescaling,
it suffices to consider the following problem.
\begin{equation}\label{elliptic-our}
\left\{
\aligned
& \Delta u  = n(n-2) u^p, \ \ u > 0 \ \ {in} \ \ \ H,\\
& \frac {\partial u}{\partial t} = -c_+u^{q} \ \ \ {on} \ \ \
\partial H
\endaligned
\right.
\end{equation}
for a positive constant $c_+$.

We first point out the specialty of this case: there might be more than two types of positive solutions.
%
%
In fact, one can check that for $x=(x', t)\in H$ and $n\geq 3$, both
$$u(x',t)= [x'^2+(d+t)^2]^{-\gamma} \quad \text{and} \quad u(x',t)= (dA^{-1}+A t)^{-2\gamma}$$
satisfy the following elliptic equation:
$$
\left\{
\aligned
& \Delta u  =2\gamma (2\gamma+2-n) u^p, \ \ u > 0 \ \ {in} \ \ \ H,\\
& \frac {\partial u}{\partial t} = - 2 d\gamma u^{p} \ \ \ {on} \ \ \
\partial H,
\endaligned
\right.
$$
where $ d>0, ~\gamma>\frac{n-2}{2}, A=[(2\gamma+2-n)/(2\gamma+1)]^{1/2}$ and $p=1+\frac{1}{\gamma},$ which implies $p<\frac{n}{n-2}.$

%
%

\medskip

Our first result can be stated as follows.

\begin{theorem}\label{Class-noncritical}
If $\ u(x) \in C^2(H) \cap C^1(\bar H)$ solves \eqref{elliptic-our} and  $n\ge 3$, $ p\geq\frac{n+2}{n-2}, 1\leq q< \frac{n}{n-2}$, then $ u=
(\frac{p-1}{2}At+B)^{-2/{(p-1)}}$, where
 $A=(\frac{2n(n-2)}{p+1})^{1/2}$ and  $B= (c_+^{-1}A)^{-(p-1)/(2q-p-1)}.$
\end{theorem}

Apparently, the supercritical power assumption (that is $p\geq\frac{n+2}{n-2}$ ) is needed in Theorem \ref{Class-noncritical} due to the above mentioned example.

\smallskip

For critical case, we have

%

\begin{theorem}\label{Class-critical}
If $\ u(x) \in C^2(H) \cap C^1(\bar H)$ solves \eqref{elliptic-our} and $n\ge 3$, $ p=\frac{n+2}{n-2}, q=\frac{n}{n-2}$.
Then for $c_+\neq n-2$, $u(x', t)=\big(\frac{\varepsilon}{|(x', t)-(x'_0, t_0)|^2-\varepsilon ^2} \big)^{(n-2)/2}$; for $c_+=n-2$, either $u(x', t)=\big(\frac{\varepsilon}{|(x', t)-(x'_0, t_0)|^2-\varepsilon ^2} \big)^{(n-2)/2}$ or $u(x',t)=u(0,t)=(2t+u(0)^{-2/(n-2)})^{-(n-2)/2}$, where $0<\varepsilon<\big( (n-2)^{-1}c_+ \big)^{2/n}, \ x'_0\in
\Bbb{R}^{n-1},$ and $t_0= - c_+(n-2)^{-1}\varepsilon^{-(n-2)/2}.$
\end{theorem}

\smallskip

If  both $ p>\frac{n+2}{n-2}$ and $q> \frac{n}{n-2}$, solutions of \eqref{elliptic-our} are complicated.  We have the following observation.

(1) If $2q\neq p+1$, then $u=(\frac{p-1}{2}At+B)^{-2/(p-1)}$ in Theorem \ref{Class-noncritical} is also a positive solution to \eqref{elliptic-our}. But there may be other solutions.

(2) If $2q=p+1$ and there is a $m\in\mathbb{N}$ satisfies $3\leq m<n, p=\frac{m+2}{m-2}$, and $ q=\frac{m}{m-2}$, then by Theorem \ref{Class-critical}, we know that there is a solution $v=v(x_1,\cdots, x_{m-1},t)$ solves
\begin{eqnarray*}
	\begin{cases}
		\Delta v=m(m-2)v^p, v>0 &\text{ in } H \ ,\\
		\frac{\partial v}{\partial t}=-cv^q & \text{ on } \ \partial H.
	\end{cases}
\end{eqnarray*}
Then $u(x_1,\cdots,x_{m-1},\cdots, x_{n-1},t)=v(x_1,\cdots, x_{m-1},t)$ satisfies equation \eqref{elliptic-our} after suitable scaling. But there may be  other solutions.


The solution set is more complicated in the case $1\leq p=q<\frac{n}{n-2}$.

\medskip

Theorem \ref{Class-noncritical} will be proved by applying  the standard moving plane method  for the equation after Kelvin transformation. The noncritical power yields  extra weight functions in the equation and/or on the boundary condition. In return, we can prove that all positive solutions only depend on the variable $t.$

However, the method of moving planes is not suitable for the critical case, basically due to the fact that the reflection plane may stop at any position. Instead,  we will use  the method of moving spheres, introduced by Li and Zhu in \cite{LZ1995}. Here, the major difference between our current work and other previous papers using the method of moving spheres is that both scenarios can happen: (1) all reflective spheres never stop, this leads to the result that the solution only depends on $t$; (2) all reflective spheres stop at a critical position, this leads to the conclusion that the solution is radially symmetric with respect to a point.

%

\medskip

Our paper is organized as follows: We will prove Theorem \ref{Class-noncritical}
by using the method of moving planes in Section 2.
The proof of Theorem \ref{Class-critical} will be presented in Section 3.  In this section, we first consider the case that all reflective spheres do not stop (Proposition \ref{lambda-notstop} below), which leads to the solutions that only depend on the last variable;
then we consider the case that  all reflective spheres do stop at certain points, which leads to the  radially symmetric solutions.

\section{Noncritical results\label{Section 2}}

In this section, we shall prove Theorem \ref{Class-noncritical}. For the sake of convenience, after rescaling we can consider the following problem:
\begin{equation}\label{elliptic-our-1}
\left\{
\aligned
& \Delta u  =  u^p, \ \ u > 0 \ \ {in} \ \ \ H,\\
& \frac {\partial u}{\partial t} = - u^{q} \ \ \ {on} \ \ \
\partial H.
\endaligned
\right.
\end{equation}
And throughout this section, we always assume that $n\ge 3$, $p\ge \frac{n+2}{n-2}$ and $1\le q <\frac n{n-2}. $

Since there is no assumption on the decay rate of $u(x)$
at infinity, as usual
we perform the Kelvin transformation on $u$, that is,  set
$$
v(x) = \frac 1{|x|^{n-2}} u\Big( \frac x {|x|^2}\Big).
$$
Then $ v(x)$ satisfies
\begin{equation}\label{kelvin}
\left\{
\aligned
 & \Delta v =|x|^\tau v^p, \ \ v(x)>0 \ \ \text{ in} \ \ H, \\
& \frac {\partial v}{\partial t}  = -\frac{1}{{|x|^{\alpha}}} v^q \ \ \ \text{ on} \ \
\partial H \setminus \{0\},
\endaligned
\right.
\end{equation}
where $\tau=p(n-2)-(n+2)>0$ and $ \alpha = n- q(n-2) >0$.
For $R>0,$ denote
$$
\aligned
B_R(x_0)&=\{x\in\mathbb{R}^n\,|\,\, |x-x_0|<R, x_0\in\mathbb{R}^n\},\\
B_R^+(x_0)&=\{x \,|\,\,|x-x_0|<R, x\in H, x_0\in H\}.
\endaligned
$$
When $x_0=0$, we simply denote
$B_R=B_R(0),\ B_R^+=B_R^+(0).$

Our goal is to obtain
some  symmetric
properties for $ v(x)$. We shall achieve this goal
by using the method of moving planes. 

Our first lemma, which  is
a modification of Lemma 2.1 in \cite{LZ1995} and Lemma 2.2 in \cite{LZ1999}, will be used to handle the possible
singular point.

\medskip
\begin{lemma}\label{singular}
Let   $v$
 $\in C^2( H )\cap
C^1( \bar H)\setminus \{0\}$ satisfy \eqref{kelvin}. Then for any $ 0< \epsilon <{\min\{ 1,
 \min_{\partial B_1^{+}\cap \partial B_1} v\}},$ we have $
v(x)\ge\frac {\epsilon}{2} $ for all  $x \in
\overline{ B_1^{+}} \setminus
\{0\}$.
\end{lemma}

\smallskip

\begin{proof}\ For $0<r<1$, we introduce
 an auxiliary function
\begin{equation}\label{lemma-2.1-0}
\varphi (x) = \frac {\epsilon}{2} - \frac {r^{n-2}
\epsilon}{|x|^{n-2}} + \frac
{\epsilon t^2 } 2,\ \ \ \  x\in {B}_1^{+} \setminus B_r^{+},
\end{equation}
and let
  $ F_1(x)= v(x)- \varphi (x)$. Then we have
\begin{equation}\label{lemma-2.1-1}
  \left\{
  \aligned
   & \Delta F_1 = |x|^{\tau}v^p - \epsilon \ \ \ \text{in} \ \
   B_{1}^{+} \setminus
   {B}_r^{+},
   \\
    &\frac {\partial F_1}{ \partial t} =- \frac{1}{|x|^{\alpha}} v^q  \ \ \  \text{on} \ \ \partial
    (\overline{{B}_{1}^{+}} \setminus \overline{ B_r^{+}})
    \cap \partial \Bbb {R}_{+}^n.
    \endaligned
    \right.
\end{equation}
    We want to show that
    \begin{equation}\label{lemma-2.1-2}
    F_1 \ge 0 \ \ \ \text{in}\ \
    \overline{{B}_{1}^{+}}\setminus \overline {B_r^{+}}.
   \end{equation}

Let
    $$
    \aligned
    & S= \{ x \ : \ |x|^{\tau} v^p - \epsilon >0 \} \cap ( B_{1}
    \setminus B_r), \\
    & S^c=\{ x \ : \ |x|^{\tau} v^p - \epsilon \le 0 \} \cap ( B_{1}
    \setminus B_r).
    \endaligned
    $$
In $ S$, since $ 1 \ge |x|>0$,  we have
    $v^p(x)> {\epsilon/|x|^{\tau}} \geq \epsilon.$
    It follows that $ v(x) > \epsilon \ge \varphi_1 (x)$ for $ x \in S$,
    that is,  $ F_1 > 0$ in $S$;
    in $S^c$, $\Delta F_1 \le 0$.
On $ \partial B_r^{+}\cap  \partial B_r$,
 $ F_1= v- ( \frac {\epsilon}{2}
 - \epsilon + \frac
{\epsilon }2t^2) \geq v> 0$; on $ \partial B_1^{+}
\cap  \partial B_1$, $ F_1= v- ( \frac
{\epsilon}{2} - r^{n-2} {\epsilon} + \frac {\epsilon
}2t^2) > v- \epsilon\geq 0$. Suppose
that \eqref{lemma-2.1-2} fails,  it follows
from the  Maximum Principle that  there exists some $ x_0=( x_0', 0)$
with $r<| x_0'|<1$ such that $F_1( x_0)=  \min_{
\overline {{B}_1^{+}} \setminus  \overline {B_r^{+}}}F_1  <0$.
Therefore $\frac {\partial F_1}{\partial t } ( x_0) \geq
0$,
which contradicts to the boundary condition in \eqref{lemma-2.1-1}.
We thus obtain  \eqref{lemma-2.1-2}.
Sending $ r \to 0$, we complete the proof of Lemma \ref{singular}.
\end{proof}


 \medskip
\begin{corollary}\label{singular-cor1}
(scaled version). \ \ Let   $v\in C^2( H )\cap
C^1( \bar H)\setminus \{0\}$
 solve \eqref{kelvin}. Then for all $0<\epsilon < {\min\{ {r_0}^{(\tau+2)/(1-p)},
 \min_{\partial B_{r_0}^{+}\cap \partial B_{r_0}} v\}},$ we have $
v(x) \geq \frac {\epsilon}{2} $ for all  $x \in
\overline{{B}_{r_0}^{+}}\setminus
\{0\}$.
\end{corollary}

 \medskip
\begin{proof}
For $r_0>0$, we obtain the result by applying Lemma \ref{singular}
to $ \bar v(x) = {r_0}^{(\tau+2)/(p-1)} v(r_0 x)$ in
$\overline{ B_1^{+}} \setminus
\{0\}$.
\end{proof}

\medskip

For $ \lambda <0$, we define
$$
\aligned
&\Sigma_{\lambda}  = \{x \ | \ t>0, \ x_1 > \lambda \}, \ T_{\lambda} = \{x \ | \ t \ge
0,\  x_1 = \lambda \}, \\
& \tilde \Sigma_{\lambda}  = \overline \Sigma_{\lambda} \setminus  \{0
\},\ \ x^{\lambda}
\  \mbox {as \ the \ reflection \ point \ of} \ x \ \mbox{ about} \ \ T_{\lambda}, \\
&v_{\lambda} (x)  = v( x^{\lambda}), \ w_{\lambda} =v(x) - v_{\lambda} (x).
\endaligned
$$
Then $ w_{\lambda}(x) $ satisfies
\begin{equation}\label{the-minus}
    \left \{
    \aligned
    & \Delta w_{\lambda} = |x|^{\tau} v^p- |x^\lambda|^{\tau} v_\lambda^p \le  c_1(x)w_{\lambda} \ \ \ \text{in} \ \ \Sigma_{\lambda}, \\
    & \frac {\partial w_{\lambda}}{ \partial t}=-\frac{1}{|x|^{\alpha}}v^q+\frac{1}{|x^\lambda|^{\alpha}}v_\lambda^q  \le - c_2(x)  w_{\lambda} \ \
    \text{on} \ \
    \partial H \cap \tilde \Sigma_{\lambda},
    \endaligned
    \right.
 \end{equation}
    where $ c_1(x)= p |x|^{\tau} \cdot \xi_1^{p-1}(x)>0$, $ c_2(x) = \frac{q}{|x^\lambda|^{\alpha}}
     \xi_2^{q-1}(x)>0$, $\xi_1 $ and  $\xi_2$ are two
    functions
    between $v_{\lambda}$ and $ v$.
 Now we are ready to use the  moving plane method.

\medskip

\begin{proposition}\label{monotonicity} There is a $N>0$, such that, for $\lambda<-N$,
  $ w_{\lambda}(x) \ge 0$
    for all $ x
    \in \tilde \Sigma_{\lambda}$.
\end{proposition}

    \medskip
\begin{proof}\
Write $z=x+(0,...,0,1).$  For a fixed positive parameter $\beta\in(0, n-2)$, we define
    $$
   \varphi_\lambda (x)=|(x_1,..., x_{n-1}, t+1)|^\beta
   w_{\lambda}(x)=|z|^\beta
   w_{\lambda}(x).   $$
   It is sufficient to prove the proposition for  $\varphi_\lambda (x) $.

   A direct calculation shows that
\begin{equation}\label{proposition 2.2-1}
     -\triangle\varphi_\lambda +\frac{2\beta}{|z|^2} z\cdot \nabla\varphi_\lambda  + c_3(x)\varphi_\lambda \geq 0 , \quad \text{in}~\Sigma_{\lambda},
 \end{equation}
where  $c_3(x)=\frac{\beta(n-2-\beta)}{|z|^2}+ c_1(x)>0;$   And
\begin{equation}\label{proposition 2.2-2}
        \frac{\partial\varphi_\lambda}{\partial t} =\frac{\beta}{|z|^2}\varphi_\lambda+|z|^\beta \frac{\partial w_\lambda}{\partial t}
    \leq
   c_4(x) \varphi_\lambda,  \ \ \quad \text{on} \ \ \partial H \cap \tilde \Sigma_{\lambda}\end{equation}
where
\begin{equation}\label{proposition 2.2-3}
        c_4(x)=\frac{\beta}{|z|^2}-c_2(x) =\frac{\beta}{|z|^2}-\frac{q}{|x^\lambda|^{\alpha}}\xi_2^{q-1}(x).
    \end{equation}

       Suppose that $\inf_{\Sigma_\lambda} \varphi_\lambda (x)<0$ for $\lambda$ sufficiently negative. First we observe that
         $$ \varphi_\lambda (x) \to 0 \quad \text{as} \quad  |x| \to \infty.$$
         And by  Lemma \ref{singular},
    we know for some $ |\lambda|$ large enough,  $
    w_{\lambda} ( x) \geq \frac{\epsilon}{2}>0$ as $x \in \overline{{B}_1^{+}(0)} \setminus \{0\}.$
   Hence, there exists  $\bar x$ such that  $
    \varphi_{\lambda} (\bar x)
    = {\displaystyle
    \min_{x
    \in  \tilde \Sigma_{\lambda}} \varphi_{\lambda}(x)} <0$.   We know from the Maximum Principle that $\bar x   \in  \partial H \cap \tilde \Sigma_{\lambda}.$

    Since    $\varphi_\lambda(\bar x)< 0$, we have
    \[
 0<v(\bar x)<\xi_2(\bar x)<v_\lambda(\bar x)= \frac 1{|\tilde x^\lambda|^{n-2}} u\Big( \frac {{\bar x} ^\lambda} {|\bar x^\lambda|^2}\Big)
< \frac{A+\circ (1)}{|\lambda|^{n-2}},
\]
where $A=\lim _{|x|\rightarrow 0}u(x)=u(0).$ It follows that
\[
 0<c_2(\bar {x})<
 \frac{q(A+\circ (1))^{q-1}}{|\lambda|^{(n-2)(q-1)+\alpha}}\rightarrow 0 \ \ \text{as} \ \lambda\rightarrow -\infty~.
\]
Thus  $ c_4(\bar x) >0$ .
        But this leads to a
    contradiction to
    the boundary condition in  \eqref{proposition 2.2-2}.
    \end{proof}
 \medskip

  We   then  can define
    \begin{equation}\label{superior}
    \lambda_0 = \sup \{ \lambda <0 \ | \ w_{\mu} (x) \ge 0 \ \ \text{in} \ \tilde
    \Sigma_{\mu} \
    \text{for ~ all} \ \ -\infty <\mu <\lambda \}.
    \end{equation}

  \medskip

\begin{proposition}\label{lamubda=0}
 $ \lambda_0 = 0$
\end{proposition}

\medskip

\begin{proof}\
    Suppose for the contrary that  $ \lambda_0<0,$ then we claim that
\begin{equation}\label{proposition 2.4-1}
    w_{\lambda_0} \equiv 0.
\end{equation}
This is a contradiction to the boundary condition in \eqref{the-minus} since $\alpha>0.$ Therefore it suffices to prove \eqref{proposition 2.4-1} under the assumption $\lambda_0 <0.$


Suppose that \eqref{proposition 2.4-1} is false, then  $w_{\lambda_0} \not \equiv 0$ satisfies
\begin{equation}\label{proposition 2.4-1-2}
\left\{
\aligned
 - \Delta w_{\lambda_0} + c_1(x)  w_{\lambda_0} \geq 0  \ \ & \text{ in} \ \ \Sigma_{\lambda_0} , \\
\frac {\partial w_{\lambda_0}}{\partial t}  \leq - c_2(x)  w_{\lambda_0}\ \ \ \ \ \ & \text{ on} \ \
\partial H \cap \tilde \Sigma_{\lambda_0} ,\\
w_{\lambda_0}\geq 0 \ \ \  \ \ \ \ \ \ \ \ \ \  \ \ \ \ \ &\text{ in} \ \ \tilde \Sigma_{\lambda_0} .
\endaligned
\right.
\end{equation}
where $ c_1(x),  c_2(x) >0$ are the same as that in \eqref{the-minus}. It follows from the Strong Maximum Principle and Hopf lemma that
 \begin{equation}\label{proposition 2.4-1-3}
\left\{
\aligned
w_{\lambda_0}(x) > 0 , \ \   \ \ \ \ \ & x\in \tilde \Sigma_{\lambda_0} \setminus T_{\lambda_0} , \\
\frac {\partial w_{\lambda_0}}{\partial x_1} (x) >0,\ \ \ \ \ \ & x\in T_{\lambda_0}\cap H .
\endaligned
\right.
\end{equation}

    The  following lemma is
    needed to deal with the possible singular point.

\medskip

\begin{lemma}\label{singularity-2}
For $ r_0\leq\min\{\frac 12 |\lambda_0|, 1 \},$   there
exists some positive constant $\gamma$ depending only on $\lambda_0$ and $ r_0$
     such that $ w_{\lambda_0}(x) >
    \gamma$ in $ B_{r_0}^+ \setminus \{ 0\}$.
\end{lemma}

 \smallskip


 We also need next lemma to show that  negative minimum point of $ w_{\lambda} $ for $ \lambda<0$  will stay uniformly bounded.  We
  postpone the proofs of both lemmas till the end of the proof of this proposition.

  \medskip

 \begin{lemma}\label{neg-min-point}
For any fixed $ 0<\delta<-\lambda_0, $   there
exists an $R_0> 1$ depending only on $n, q, v,$
     such that if for any $\lambda \in (\lambda_0, \lambda_0+\delta),$ $ w_{\lambda}(x^0) = {\displaystyle \min_{\tilde \Sigma_{\lambda}} w_{\lambda}}
    <0,$ then  $|x^0|\leq R_0$.
\end{lemma}

   \medskip
We now continue the proof of Proposition \ref{lamubda=0}. By the definition of $ \lambda_0$, we know that  there exists
 a sequence  $ \lambda_k \to \lambda_0 $
    with $ \lambda_k > \lambda_0$ such that
    $$
    \inf_{\tilde \Sigma _{\lambda_k}} w_{\lambda_k} <0.
    $$
 Clearly, $ {\displaystyle
    \lim_{|x|\rightarrow +\infty}} w_{\lambda_k} (x)=0.$
From Lemma \ref{singularity-2} and the continuity of $v(x)$ away from the origin, we know that for $k$ large enough, there exists $\gamma>0$ such that
$$
 w_{\lambda_k}(x) >
   \frac{ \gamma}{2} \quad \text{in}~~  B_{r_0}^+ \setminus \{0\}.
$$
It follows that there exists  a point $ x^k = ((x^k)', t^k) \in \tilde
    \Sigma_{\lambda_k} \setminus {B_{r_0}^+},$  such that
\begin{equation}\label{proposition 2.4-2}
    w_{\lambda_k}(x^k) = {\displaystyle \min_{\tilde \Sigma_{\lambda_k}} w_{\lambda_k}}
    <0.
\end{equation}

Moreover,  Lemma \ref{neg-min-point} implies $|x^k|\leq R_0.$
From equation \eqref{the-minus} and the  Maximum Principle, we know that $x^k$ cannot be a interior point. Thus $x^k$ must be on the lateral boundary:
$$
\{(x_1,..., x_{n-1}, t): \  t=0,\  x_1>\lambda_k, \   r_0 \leq|x|\leq R_0 \} .
$$
Hence,
\begin{equation}\label{proposition 2.4-3}
\frac{\partial w_{\lambda_k}}{\partial x_1}  (x^k)=0, \quad (\text{since  along }\ x_1 \ \text{direction is the tangential direction}).
 \end{equation}



Therefore, there is a sequence of $x^k,$ still denoted by $x^k$, such that  ${\displaystyle \lim_{k\rightarrow \infty}} x^k = x^0.$ By the continuity of $w_{\lambda},$ it holds that $0\leq w_{\lambda_0}(x^0)={\displaystyle \lim_{k\rightarrow \infty}} w_{\lambda_k}(x^k)\leq 0,$  that is, $ w_{\lambda_0}(x^0)=0,$ thus $ x^0=(x^0_1,\ldots, x^0_{n-1}, t^0) \in T_{\lambda_0}\cap \{t=0\}.$ From \eqref{proposition 2.4-3}, we have
\begin{equation}\label{proposition 2.4-4}
 w_{\lambda_0} (x^0)=0, \quad \frac{\partial w_{\lambda_0}}{\partial x_1}  (x^0)=0.
\end{equation}
It follows that from \eqref{proposition 2.4-1-3} and \eqref{proposition 2.4-4} that $|x^0_1|=\lambda_0$ and $t^0=0.$

We need another  lemma to handle  the partial derivative about the first variable $x_1$ at the corner point.
\smallskip
\begin{lemma}\label{normal-derivative-0}\ Suppose \eqref{proposition 2.4-1-2} and \eqref{proposition 2.4-1-3} hold, then $\frac {\partial w_{\lambda_0}}{\partial x_1} (x^0) >0$ for all $ x^0=(x^0_1,\ldots, x^0_{n-1} , 0)$ with $x^0_1=\lambda_0.$
\end{lemma}
\medskip

Using Lemma \ref{normal-derivative-0}, we reach a contradiction due to \eqref{proposition 2.4-4}.

We are left to prove above three lemmas.


\smallskip
\noindent{\bf Proof of Lemma \ref{singularity-2}.}\ For $0<r<r_0, ~x\in {B_{r_0}^{+}} \setminus {B_r^{+}}$ and $\varepsilon>0$ being a positive constant satisfying $\min_{\partial {B_{r_0}^{+}}} w_{\lambda_0} \geq \varepsilon,$ let $\varphi (x)$ be defined  in \eqref{lemma-2.1-0},
and
  $ F_2(x)= w_{\lambda_0}(x)- \varphi (x)$. Then $ F_2(x)$ satisfies
\begin{equation}\label{lemma-singularity-2-1}
  \left\{
  \aligned
   & \Delta F_2 =  |x|^{\tau} v^p- |x^{\lambda_0}|^{\tau} v_{\lambda_0}^p-
   \epsilon, \quad\quad\ \
   x\in {B_{r_0}^{+}} \setminus
   {B_r^{+}},   \\
    &\frac {\partial F_2}{ \partial t}=-\frac{1}{|x|^{\alpha}}v^q(x)+\frac{1}{|x^{\lambda_0}|^{\alpha}}v_{\lambda_0}^q(x),  \quad x\in
    \partial({\overline{B_{r_0}^{+}}} \setminus { \overline{B_r^{+}}})
    \cap \partial H.
    \endaligned
    \right.
\end{equation}
Let
$$
\aligned
   & S=\{x:\ |x|^{\tau} v^p- |x^{\lambda_0}|^{\tau} v_{\lambda_0}^p-
   \epsilon>0 \}\cap ({B_{r_0}^{+}} \setminus
   {B_r^{+}}),\\
   &S^c=\{x:\ |x|^{\tau} v^p- |x^{\lambda_0}|^{\tau} v_{\lambda_0}^p-
   \epsilon \leq 0 \}\cap ({B_{r_0}^{+}} \setminus
   {B_r^{+}}).
   \endaligned
$$

In $S,$ $|x|^{\tau}( v^p-v_{\lambda_0}^p) \geq |x|^{\tau} v^p- |x^{\lambda_0}|^{\tau} v_{\lambda_0}^p>\varepsilon.$ It follows from $ |x|^{\tau}( v^p-v_{\lambda_0}^p) = p |x|^{\tau} \cdot \xi_1^{p-1}(x) w_{\lambda_0} $  that $ w_{\lambda_0}>\gamma$ for some $\gamma>0$ depending on $r_0$ and $\lambda_0.$

On the other hand, in $S^c,$ $\Delta F_2\leq 0.$ On $\partial {B_{r_0}^{+}}\cap \partial{B_{r_0}} ,$  $F_2(x)\geq \varepsilon-(\frac{\varepsilon}{2}-\frac{r^{n-2}\varepsilon}{|x|^{n-2}}+\frac{\varepsilon t^2}{2})>0;$
on $\partial {B_r^{+}}\cap \partial{B_r} , $ $F_2(x)\geq \varepsilon-(\frac{\varepsilon}{2}-\varepsilon +\frac{\varepsilon t^2}{2})>0.$ Suppose there exists $x_0 \in S^c,$ such that
$
F_2(x_0)=\min F_2(x)<0,$
then from the Maximum Principle we know that $x_0 \in \partial ({B_{r_0}^{+}} \setminus
   {B_r^{+}}) \cap \partial H. $
Then $\frac {\partial F_2}{ \partial t}(x_0)\geq 0$, which contradicts to the boundary condition of \eqref{lemma-singularity-2-1}. Thus  $F_2 \geq 0 $ in   $ S^c$  also.

Sending  $ r \to 0$, we obtain Lemma \ref{singularity-2}.
%

\medskip
\noindent {\bf Proof of Lemma \ref{neg-min-point}.}\ Choose a test function
 $$
 \phi(x)=(n-2)^{(n-2)/2}(\frac{\sigma}{|x'|^2+(t+\sigma)^2})^{(n-2)/2},
 $$
 where $\sigma$ satisfies $(n-2)\sigma>q \gamma_0^{q-1}$, and $\gamma_0$ is a positive constant satisfying $ 1/{\gamma_0}\leq |x|^{n-2} v(x)\leq \gamma_0$ for  $|x|\geq1.$
 Let$\bar{w}_\lambda(x)=\frac{w_\lambda(x)}{\phi(x)},$
 then it is a straight forward calculation to verify that
 \begin{equation}\label{neg-min-point-1}
     -\triangle \bar{w}_\lambda - 2(\nabla \phi/\phi ) \cdot \nabla {\bar{w}_\lambda}  + c_1(x) {\bar{w}_\lambda} \geq 0 , \quad \text{in}~\Sigma_{\lambda}\setminus B_{R_0}^+,
 \end{equation}
where  $c_1(x)>0$ is the same as that in \eqref{the-minus}.   And
\begin{equation}\label{neg-min-point-2}
        \frac{\partial \bar{w}_\lambda}{\partial t} =
   \bar{c}(x) \bar{w}_\lambda,   \quad \text{on} \ \ (\tilde{\Sigma}_{\lambda}\setminus B_{R_0}^+)\cap \partial H
   \end{equation}
where
\begin{equation*}\label{neg-min-point-3}
        \bar{c}(x)=g^{n/(n-2)}-\frac{v^q(x)/|x|^\alpha-v_\lambda^q(x)/|x^\lambda|^\alpha}{w_\lambda(x)} \quad \text{on} \ \ \partial H \cap (\tilde \Sigma_{\lambda}\setminus B_{R_0}).
    \end{equation*}
One can show that $ \bar{c}(x)>0$ on $\partial H \cap (\tilde \Sigma_{\lambda}\setminus B_{R_0})$ 
when $w_\lambda(x)<0$ and  $\alpha=n-q(n-2)>0$ (see, for example, Ou \cite{Ou1996}).
 We thus obtain Lemma \ref{neg-min-point}  from \eqref{neg-min-point-1} and \eqref{neg-min-point-2} by using the Maximum Principle.

\medskip
\noindent{\bf Proof of Lemma \ref{normal-derivative-0}.} Without loss of generality, we assume $\lambda_0=-1.$ Set $\Omega_1=\{x=(x_1, \ldots, x_{n-1}, t) | -1<x_1<-1/2, \  x_2^2+\ldots+x_{n-1}^2+t^2 <a\}$ with $0<a \in (0, 1) $ to be chosen later.
Due to the continuity of $v$ in $H\setminus \{0\},$ there exists a positive constant $C_1$ such that
 \begin{equation}\label{lemma 2.6-1}
    v_{\lambda_0}(x)<C_1<+\infty \ \ \ \text{for}\ \
    x\in \overline{\Omega}_1.
   \end{equation}
Let
$$
h_1(x)=\varepsilon ( \frac{1}{x_1^2}-1)(\frac{t^2 \mu}{2}+ 1-\mu),\ \ \  x\in \Omega_1,
$$
where $0<\varepsilon,\mu<1$ will be chosen later. Let $A_1(x)= w_{\lambda_0}(x)- h_1(x).$ Then $A_1(x)$ satisfies the following equation:
\begin{equation}\label{lemma 2.6-2}
  \left\{
  \aligned
   \Delta A_1(x) &=|x|^{\tau} v^p- |x^{\lambda_0}|^{\tau} v_{\lambda_0}^p- \Delta h_1(x) \quad\quad \text{in} \ \Omega_1\\
    \frac {\partial A_1}{ \partial t} &= -\frac{1}{|x|^{\alpha}}v^q+\frac{1}{|x^{\lambda_0}|^{\alpha}}v_{\lambda_0}^q \leq 0  \quad\quad  \text{on} \ \
   \partial \Omega_1  \cap \partial H,
    \endaligned
    \right.
\end{equation}
and
\begin{equation}\label{lemma 2.6-2-1}
\Delta h_1(x)=6\varepsilon x_1^{-4}(\frac{t^2 \mu}{2}+ 1-\mu)+\varepsilon \mu(x_1^{-2}-1).
\end{equation}

For suitably chosen $\varepsilon$ and $\mu,$ we want to show
\begin{equation}\label{lemma 2.6-3}
A_1(x)= w_{\lambda_0}(x)- h_1(x)\geq 0, \quad \forall x \in \Omega_1.
\end{equation}
Using \eqref{proposition 2.4-1-3}, 
we can choose $\varepsilon_0>0$ small enough, such that for all $0<\varepsilon<\varepsilon_0,$ we have $A_1(x)\geq 0$ on ${\partial \Omega_1} \cap \{\{x_1=-1/2\}\cup \{ x_2^2+\ldots+x_{n-1}^2+t^2 =a \}\}.$ Also, from the construction of $h_1$ we know $A_1(x)=0$ on $\partial {\Omega_1} \cap \{x_1=-1\}.$ Suppose the contrary to \eqref{lemma 2.6-3}; there exists some $\bar{x}=(\bar{x}_1,\ldots, \bar{x}_{n-1}, \bar{t})\in \overline{\Omega_1} $ such that
\begin{equation}\label{lemma 2.6-4}
A_1(\bar{x})= \min_{\overline{\Omega}_1}A_1< 0.
\end{equation}
From the above and the boundary condition of \eqref{lemma 2.6-2}, we have $\bar{t}>0, \bar{x}_2^2+\ldots+\bar{x}_{n-1}^2+\bar{t}^2<a,\ -1< \bar{x}_1<-1/2.$ Thus
\begin{equation}\label{lemma 2.6-5}
\Delta A_1(\bar{x})\geq 0.
\end{equation}
From \eqref{lemma 2.6-4}, we have $v(\bar{x})-v_{\lambda_0}(\bar{x})-h_1(\bar{x})<0,$ and then
\begin{equation}\label{lemma 2.6-6}
v(\bar{x})<C_2<+\infty
\end{equation}
for some constant $C_2$ depending only on $C_1.$ By \eqref{lemma 2.6-1}, \eqref{lemma 2.6-6} and the Mean Value Theorem we have $|\bar{x}|^{\tau} v^p(\bar{x})- |\bar{x}^{\lambda_0}|^{\tau} v_{\lambda_0}^p(\bar{x}) \leq C_3 w_{\lambda_0}(\bar{x}) $ for some positive constant $C_3$ depending only on $C_1, C_2$ and $\lambda_0.$ Hence, it follows from \eqref{lemma 2.6-2}, \eqref{lemma 2.6-2-1} and \eqref{lemma 2.6-5} that
\begin{equation}\label{lemma 2.6-7}
C_3 w_{\lambda_0}(\bar{x})\geq \Delta h_1(\bar{x})> \varepsilon \mu (\frac{1}{x_1^2}-1).
\end{equation}
Again by \eqref{lemma 2.6-4},
\begin{equation}\label{lemma 2.6-8}
  w_{\lambda_0}(\bar{x}) < h_1(\bar{x}) < \varepsilon  (\frac{1}{x_1^2}-1)(\frac{\mu }{2C_3}+ 1-\mu),
\end{equation}
where  $\bar{t}^2<a$, and we have chosen $0<a \leq {C_3}^{-1}.$ Therefore, we choose $a\leq \min \{  {C_3}^{-1}, 1\}$ such that both \eqref{lemma 2.6-7} and \eqref{lemma 2.6-8} hold.  Combining \eqref{lemma 2.6-7} and \eqref{lemma 2.6-8}, we have
$$
 \varepsilon  C_3 (\frac{1}{x_1^2}-1)(\frac{\mu }{2C_3}+ 1-\mu)> \varepsilon \mu (\frac{1}{x_1^2}-1);
$$
i. e., $\mu<\frac{2 C_3}{1+2 C_3}.$ If $\mu$ is chosen in such a way that $\mu>\frac{2 C_3}{1+2C_3}$ from the beginning, we reach a contradiction. Thus \eqref{lemma 2.6-3} holds. Since we also know that $A_1(x^0)=0,$ we have
$$
\frac {\partial A_1}{\partial {x_1}} (x^0)\geq 0.
$$
It follows from a direct computation that
$$
\frac {\partial w_{\lambda_0}}{\partial {x_1}}(x^0) =\frac {\partial A_1}{\partial {x_1}} (x^0)+\frac {\partial h_1}{\partial {x_1}} (x^0) \geq \frac {\partial h_1}{\partial {x_1}} (x^0)=2 \varepsilon (1-\mu) >0.
$$
Hence, Lemma \ref{normal-derivative-0} is established.

We hereby complete the proof of Proposition \ref{lamubda=0}.
\end{proof}

\medskip
{\bf Proof of Theorem \ref{Class-noncritical}.}\ From Proposition \ref{lamubda=0}, we know that $v(-x_1, x_2,  \cdots, t)\le v(x_1, x_2,  \cdots, t)$ for $x_1 \ge 0$. Similarly, if we  move planes from
the positive direction of
$ x_1$, we see that $v(-x_1, x_2,  \cdots, t)\ge v(x_1, x_2,  \cdots, t)$ for $x_1 \ge 0$. Thus, $ v(x', t)$ is symmetric with
respect to $x_1$.
Clearly the above argument can be  applied to any direction
perpendicular to $t$-axis, therefore we conclude  that
$v(x',t) =v(|x'|, t).$
It follows that  $u(x',t)=u(|x'|, t)$
due to the inverse
Kelvin transformation. Since we can choose the origin arbitrarily on
the hyperplane $t=0$, it is easy to see that $ u(x',t)$ is independent
of  $ x'$.  \eqref{elliptic-our} is reduced to
the following ordinary differential equation:
\begin{equation}\label{ordinary}
\left\{
\aligned
&u''(t)  = n(n-2) u^p(t), \ \ u > 0 ,\ \  t>0 ,\\
& u'(0) = -c_+ u^q(0).
\endaligned
\right.
\end{equation}
It is easy to check that \eqref{ordinary} has the following unique  positive solution
$$
u=
(\frac{p-1}{2}At+B)^{-2/{(p-1)}}
$$
where $A=(\frac{2n(n-2)}{p+1})^{1/2}$ and  $B= (c_+^{-1}A)^{-(p-1)/(2q-p-1)}.$
This completes the proof of Theorem \ref{Class-noncritical}.

\section{Critical results\label{Section 3}}

In this section, we shall deal with the critical case and prove Theorem 1.2.

Assume $\ u(x) \in C^2(H) \cap C^1(\bar H)$ solves \eqref{elliptic-our}. Throughout this whole section, we assume that $ p=\frac{n+2}{n-2}, q=\frac{n}{n-2}$ (and $n \ge 3$).

For any given point $b \in \partial H= \Bbb{R}^{n-1}, $ we define the Kelvin transformation of $u$ centered at $b$ by
$$
v_b(x) = \frac 1{|x|^{n-2}} u_b\Big( \frac x {|x|^2}\Big),
$$
where $u_b(x)=u(x'+b, t).$ If $b=0$ is the origin, for simplicity, we denote $v(x)=v_0(x)$.
Then $ v(x)$ satisfies
\begin{equation}\label{kelvin-2}
\left\{
\aligned
 & \Delta v =n(n-2) v^p, \ \ v(x)>0 \ \ \text{ in} \ \ H, \\
& \frac {\partial v}{\partial t}  = - c_+ v^q \ \ \ \text{ on} \ \
\partial H \setminus \{0\}.
\endaligned
\right.
\end{equation}


For $\lambda>0$ and $b\in \partial H,$ define
$$
w_{\lambda, b}(x) = v_b(x)-\frac {\lambda^{n-2}}{|x|^{n-2}} v_b\Big( \frac {\lambda^2 x} {|x|^2}\Big).
$$
In the rest of this section,  we always write $v_\lambda(x)=(\lambda^{n-2}/|x|^{n-2}) v( \lambda^2 x/|x|^2),$ $w_{\lambda}(x)=v(x)-v_\lambda(x).$ Clearly, $w_\lambda$ satisfies
\begin{equation}\label{kelvin-minus}
\left\{
\aligned
 & \Delta w_\lambda =n(n-2) (v^p- v_\lambda ^p)=c_5(x)  w_\lambda \ \ \ \  \ \ \text{ in} \ \ B_\lambda^+, \\
& \frac {\partial w_\lambda}{\partial t}  = -c_+ (v^q- v_\lambda ^q)=- c_6(x)  w_\lambda \ \ \ \text{ on} \ \
\partial B_\lambda^+\cap\partial H \setminus \{0\},
\endaligned
\right.
\end{equation}
where $ c_5(x)= n(n+2) \xi_3(x)^{4/(n-2)}$, $ c_6(x) = c_+(n-2)^{-1}n
     \xi_4(x)^{2/(n-2)}$, $\xi_3 $ and  $\xi_4$ are two
    functions
    between $v_{\lambda}$ and $ v$.

\medskip
We need to establish a lemma which  will be used to handle the possible singular point.

\begin{lemma}\label{singular-3}
 Let   $v\in C^2( H )\cap
C^1( \bar H)\setminus \{0\}$
 solve \eqref{kelvin-2}. Then for any $0<\epsilon <
{\min\{ R^{(2-n)/2},
 \min_{\partial B_R^{+}\cap \partial B_R} v\}},$ we have $
v(x) \geq \frac {\epsilon}{2} $ for all  $x \in
\overline{{B}_R^{+}}\setminus
\{0\}$.
\end{lemma}

\medskip
\begin{proof}
When $R=1$, this lemma can be proved as that of Lemma \ref{singular} with $\tau=\alpha=0$. For more general $R$, it follows  easily from
applying  the result for $R=1$ to $ \bar v(x) = R^{
(n-2)/2} v(Rx)$.
\end{proof}

\medskip
Next we will prove Theorem \ref{Class-critical} by the moving sphere method.

First we establish a proposition which makes it possible to start moving the spheres.

\smallskip
\begin{proposition}\label{lambda-infinite}
For $\lambda$ large enough, $ w_\lambda (x)\geq 0$ for all $x\in B_\lambda^+ \setminus \{0\}.$
\end{proposition}

\begin{proof}\ We prove this proposition by three steps as that in \cite{LZ1995}.

Step 1. Similar to the proof of Proposition 2.1 in \cite{LZ1995},  we have: there exists $R_0>0$ large enough, such that for all $R_0\leq |x|\leq \lambda/2,$ $ w_\lambda (x)\geq 0.$

Step 2.  Let $R_1\geq R_0$ and $R_1\leq \lambda/2\leq|x|\leq \lambda.$  We claim that $ w_\lambda (x)\geq 0.$\\
 To see this, we define
 $$
   \varphi_\lambda (x)=|y|^\beta w_{\lambda}(x)
    $$
    for  $\beta \in (0, n-2),  \ y=x+(0,..., 0, \lambda/4).$  Easy to check that $\varphi_\lambda $ satisfies
\begin{equation}\label{proposition 3.2-1}
 \left\{
\aligned
 &  -\triangle\varphi_\lambda +\frac{2\beta}{|y|^2} y\cdot \nabla\varphi_\lambda  + (c_5(x)+\frac{\beta(n-2-\beta)}{|y|^2})\varphi_\lambda = 0 , \quad \text{in}~B_{\lambda}^+, \\
& \frac {\partial \varphi_\lambda}{\partial t}  = (\frac{\beta\lambda}{4|y|^2}- c_6(x))  \varphi_\lambda \ \ \ \text{ on} \ \
\partial B_\lambda^+\cap\partial H \setminus \{0\}.
\endaligned
\right.
 \end{equation}
If there exists $x_0$ with $\lambda/2\leq |x_0|\leq \lambda,$ such that  $\varphi_{\lambda}(x_0) = {\displaystyle \min_{ \lambda/2\leq |x|\leq \lambda } \varphi_{\lambda}(x)}
    <0,$ then from step 1 and the definition of $w_\lambda,$ we know $x_0 \notin \{|x|=\lambda/2\}\cup \{|x|=\lambda\}.$  The Maximum Principle yields that $x_0\in \partial H.$ As in the proof of  Proposition 2.1 in \cite{LZ1995}, we can show that $c_6(x_0)\leq C/|x_0|^2$ with $C>0.$ Moreover, for $\lambda$ large enough and $y_0=x_0+(0,..., 0, \lambda/4),$ we have
$$
\frac{\beta\lambda}{4|y_0|^2}-c_6(x_0)>0, \quad x_0\in \partial H.
$$
This contradicts to the boundary condition in \eqref{proposition 3.2-1}.

Step 3. We claim: there exists $R_2\geq R_1, $ such that for $\lambda \geq R_2,$  $ w_\lambda (x)\geq 0$ for $x\in B_{R_0}^+\setminus \{0\}.$

From the fact that ${\displaystyle \lim_{ |y|\rightarrow +\infty }|y|^{n-2}v(y)=u(0) },\ |\frac{\lambda^2x}{|x|^2}|\geq \frac{\lambda^2}{R_0}\geq \frac{R_2^2}{R_0}$ and Lemma \ref{singular-3}, we have
$$
w_{\lambda}(x) = v(x)-\frac{1}{\lambda^{n-2}}\big(\mid\frac {\lambda^2x}{|x|^2}\mid^{n-2} v( \frac {\lambda^2 x} {|x|^2})\Big)>\frac{\varepsilon}{2}-\frac{u(0)+\circ (1)}{\lambda^2}.
$$
So Step 3 follows easily as  $\lambda$ becomes large.  Proposition \ref{lambda-infinite} is proved.
\end{proof}

Now for any $b \in \partial H= \Bbb{R}^{n-1} $, we define
$$
\lambda_b=\inf\{\lambda>0|w_{\mu, b}(x)\geq 0\  \text{in} \  \overline{B_\mu^+}\setminus \{0\}\ \text{for all }\  \lambda<\mu<\infty\}.
$$
\begin{proposition}\label{lambda-notstop} Assume $ u(x) \in C^2(H) \cap C^1(\bar H)$ solves \eqref{elliptic-our}.  If  $\lambda_b=0$ for all $b \in \partial H= \Bbb{R}^{n-1}, $
then $c_+=n-2$ and $
u(x)=u(x', t)=u(0, t)=u(t)=(2t+u(0)^{-2/(n-2)})^{-(n-2)/2},$ for  all  $x=(x',t)\in H.
$

\end{proposition}



In order to prove  Proposition \ref{lambda-notstop}, we need the following technical Li-Zhu lemma which appeared first  in \cite[Lemma 2.2]{LZ1995}.

\begin{lemma}\label{keylemma1} (Lemma~$2.2$~ in~\cite{LZ1995}) Suppose $f\in C^1(H)$  satisfies:  for any  $b\in \partial H,\ \lambda>0,$
\begin{equation}\label{LZ-1}
f_b(x)-\frac {\lambda^{n-2}}{|x|^{n-2}} f_b\Big( \frac {\lambda^2 x} {|x|^2}\Big)\geq 0, \quad \forall x\in B_\lambda^+,
\end{equation}
where $f_b(x)=f(x'+b,t), \forall x=(x',t)\in H.$ Then
$f(x)=f(x',t)=f(0,t)$ for all $x\in H.$
\end{lemma}

{\bf Proof of Proposition \ref{lambda-notstop}.}\ If $\lambda_b=0$ for all $b \in \partial H= \Bbb{R}^{n-1}, $ then $u(x', t) \in C^1(\bar H)$ satisfies \eqref{LZ-1}. We know from Lemma \ref{keylemma1} that  $u(x)=u(x',t)=u(0,t)=u(t),$ thus
$$
\left\{
\aligned
&u''(t)  =n(n-2) u^p(t), \ \ u > 0 ,\ \  t>0 ,\\
& u'(0) = -c_+ u^q(0).
\endaligned
\right.
$$
It is easy to see that
$$
\big(u'(t) \big )^2 =(n-2)^2u^{{2n}/{(n-2)}}(t)+\big(c_+^2-(n-2)^2\big)u^{{2n}/{(n-2)}}(0).
$$

When $c_+=n-2,$ $u(t)$ satisfies
$$
\left\{
\aligned
&\big(u'(t) \big )^2 =(n-2)^2u^{{2n}/{(n-2)}}(t), \ \ u > 0 ,\ \  t>0 ,\\
& u'(0) = -(n-2) u^{{2n}/{(n-2)}}(0),
\endaligned
\right.
$$
We then obtain $u( t)=
(2t+u(0)^{-2/(n-2)})^{-(n-2)/2}.$

When $c_+> n-2,$ it holds
$$u'(t)< -\sqrt{ (c_+^2-(n-2)^2)u^{{2n}/{(n-2)}}(0)},$$
then there is no global positive solution.

When $0<c_+<n-2,$ since $u''(t)>0$, we know that $u(t)$ satisfies
$$
u'(t)=\left\{
\aligned
& - \sqrt{(n-2)^2u^{\frac{2n}{n-2}}(t)+(c_+^2-(n-2)^2)u^{\frac{2n}{n-2}}(0)}, \ \ 0\leq t < t_0,\\
& \, 0, \quad \quad \quad \quad\quad \quad\quad \quad \quad \quad\quad \quad\quad \quad\quad \quad\quad \quad\quad \quad  t = t_0,\\
&  \sqrt{(n-2)^2u^{\frac{2n}{n-2}}(t)+(c_+^2-(n-2)^2)u^{\frac{2n}{n-2}}(0)},  \quad \quad t > t_0,
\endaligned
\right.
$$ for some positive $t_0$. This means $u(t)\rightarrow+\infty$ as $t\rightarrow+\infty$. Thus there exists $t_1>t_0$, such that for $t>t_1$, $(n-2)^2u^{\frac{2n}{n-2}}(t)-((n-2)^2-c_+^2)u^{\frac{2n}{n-2}}(0)>\frac{1}{2}(n-2)^2u^{\frac{2n}{n-2}}(t).$
So, we have that for $t>t_1$,
\begin{eqnarray*}
	t-t_1&=&\int_{t_1}^{t}\frac{u'(s)ds}{ \sqrt{(n-2)^2u^{\frac{2n}{n-2}}(s)+(c_+^2-(n-2)^2)u^{\frac{2n}{n-2}}(0)}}\\
	&=&\int_{u(t_1)}^{u(t)}\frac{du}{ \sqrt{(n-2)^2u^{\frac{2n}{n-2}}+(c_+^2-(n-2)^2)u^{\frac{2n}{n-2}}(0)}}\\
	&\leq&\int_{u(t_1)}^{u(t)}\frac{\sqrt{2}du}{(n-2)u^{\frac{n}{n-2}}}\\
	&=&\frac{\sqrt{2}}{2}u^{-2/(n-2)}(t_1)-\frac{\sqrt{2}}{2}u^{-2/(n-2)}(t),
\end{eqnarray*}
which impies
$$u^{-2/(n-2)}(t)\leq u^{-2/(n-2)}(t_1)-\sqrt{2}(t-t_1).$$
But as $t\rightarrow +\infty,$ $ u^{-2/(n-2)}(t_1)-\sqrt{2}(t-t_1)\rightarrow -\infty$,
contradiction! So there is no positive global solution.
We hereby complete the proof of Proposition \ref{lambda-notstop}.

\medskip

Next, we consider the other possibility, that is,  $\lambda_b>0$ for some $b \in \partial H= \Bbb{R}^{n-1}. $

\begin{proposition}\label{lambda-stop}  Assume $ u(x) \in C^2(H) \cap C^1(\bar H)$ solves \eqref{elliptic-our}.
If  $\lambda_b>0$ for some $b \in \partial H= \Bbb{R}^{n-1}, $ then
$$
w_{\lambda_b,  b}(x)\equiv 0, \quad \forall x=(x',t)\in H.
$$
\end{proposition}
\begin{proof}
From the properties of the Kelvin transformation, we only need to prove this proposition for $x\in B_{\lambda_b}^+(b).$ Without loss of generality, we assume $b=0.$ Suppose the contrary to Proposition \ref{lambda-stop},  then
$w_{\lambda_0}\not \equiv 0$ satisfies
\begin{equation}\label{proposition 3.5-1}
\left\{
\aligned
 - \Delta w_{\lambda_0} + c_5(x)  w_{\lambda_0} =0  \ \ & \text{ in} \ \ B_{\lambda_0}^+, \\
\frac {\partial w_{\lambda_0}}{\partial t}  = - c_6(x)  w_{\lambda_0}\ \ \ \ \ \ & \text{ on} \ \
\partial B_{\lambda_0}^+\cap\partial H \setminus \{0\},\\
w_{\lambda_0}\geq 0 \ \ \  \ \ \ \ \ \ \ \ \ \  \ \ \ \ \ &\text{ in} \ \ {\overline{B_{\lambda_0}^+}}\setminus \{0\}.
\endaligned
\right.
\end{equation}
where $ c_5(x),  c_6(x) >0$ are the same as that in \eqref{kelvin-minus}. It follows from the Strong Maximum Principle and the Hopf lemma that
 \begin{equation}\label{proposition 3.5-2}
\left\{
\aligned
w_{\lambda_0}(x) > 0 , \ \   \ \ \ \ \ & x\in \overline{B_{\lambda_0}^+} , \, 0<|x|<\lambda_0 , \\
\frac {\partial w_{\lambda_0}}{\partial \nu} (x) >0,\ \ \ \ \ \ & x\in \partial B_{\lambda_0}^+\cap H ,
\endaligned
\right.
\end{equation}
where $\nu$ denotes the inner normal of the sphere $\partial B_{\lambda_0}.$ The following lemma is needed to deal with the possible singular point, and we
postpone its proof till the end of the proof of the proposition.

\medskip
\begin{lemma}\label{singularity-4}\ There exists a positive constant $\gamma=\gamma (\lambda_0)>0 $ such that $ w_{\lambda_0}(x) >
    \gamma$ for $x\in {\overline{B_{{\lambda_0}/2}^+}} \setminus\{ 0\} .$
\end{lemma}

\medskip
We continue the proof of Proposition \ref{lambda-stop}. From the definition of $\lambda_0,$ we know that there exists a sequence $\lambda_k\rightarrow \lambda_0$ with $\lambda_k< \lambda_0,$ such that
$$
\inf_{ {\overline{B_{\lambda_k}^+}}\setminus \{0\}} w_{{\lambda_k}}<0.
$$
From Lemma \ref{singularity-4} and the continuity of $ w_{\lambda}$ away from the origin, it follows that for $k$ large enough, there exists $x^k=((x^k)', t^k)\in {\overline{B_{\lambda_k}^+}}\setminus B_{{\lambda_0}/2}^{+}$ such that
\begin{equation}\label{proposition 3.5-3}
w_{\lambda_k}( x^k)=  \min_{
{\overline {B_{\lambda_k}^+}} \setminus  \{0\}}w_{\lambda_k}( x)  <0.
\end{equation}
It is clear that $\frac{\lambda_0}{2}<|x^k|<\lambda_k$, And, due to the Strong Maximum Principle, $t^k=0.$ Hence, $\frac{\partial w_{\lambda_k}}{\partial x_i}( x^k)=0 \ (i=1, \cdots, n-1)$ and $\frac{\partial w_{\lambda_k}}{\partial t}( x^k)\geq 0.$ After passing to a subsequence (still denoted as $x^k$), $x^k\rightarrow x_0=(x_0', 0).$ It follows that
\begin{equation}\label{proposition 3.5-4}
w_{\lambda_0}( x_0)=0,\ \ \frac{\partial w_{\lambda_0}}{\partial x_i}( x_0)=0, \ (i=1, \cdots, n-1).
\end{equation}
We know from \eqref{proposition 3.5-2} and \eqref{proposition 3.5-4} that $|x_0'|=\lambda_0.$

The following lemma is needed to handle the normal derivative, and we
postpone its proof till the end.

\medskip
\begin{lemma}\label{normal-derivative}\ Suppose \eqref{proposition 3.5-1} and \eqref{proposition 3.5-2} hold, then $\frac {\partial w_{\lambda_0}}{\partial \nu} (x_0) >0$ for all $ x_0=(x_0', 0)$ and $|x_0'|=\lambda_0.$
\end{lemma}

\medskip
From Lemma \ref{normal-derivative}, we reach a contradiction due to \eqref{proposition 3.5-4}.

We are left to prove above two lemmas

\smallskip

\noindent{\bf Proof of Lemma \ref{singularity-4}.}\ Using \eqref{proposition 3.5-2}, we have $\min_{{\partial B_{{\lambda_0}/2}^+}\cap {\partial B_{{\lambda_0}/2}}} w_{{\lambda_0}}\geq \varepsilon$ for some $0<\varepsilon <1.$  Without loss of generality, we assume $\lambda_0=2.$ For $0<r<1$ and $x\in {B}_1^{+} \setminus B_r^{+},$ let $\varphi (x)$ be the same function given  in \eqref{lemma-2.1-0},
and
  $ F_3(x)= w_{\lambda_0}(x)- \varphi (x)$. Then $ F_3(x)$ satisfies
\begin{equation}\label{lemma-singularity-4-1}
  \left\{
  \aligned
   & \Delta F_3 = n(n-2)(v^p -v_{\lambda_0}^p)-
   \epsilon \ \ \ \text{in} \ \
   B_{1}^{+} \setminus
   {B}_r^{+},   \\
    &\frac {\partial F_3}{ \partial t} =- c_+ (v^q -v_{\lambda_0}^q)  \ \ \ \   \ \ \  \text{on} \ \ \partial
    ({{B}_{1}^{+}} \setminus { B_r^{+}})
    \cap \partial H.
    \endaligned
    \right.
\end{equation}
If $n(n-2)(v^p -v_{\lambda_0}^p)-
   \epsilon \geq 0,$ we have proved Lemma \ref{singularity-4} with $\gamma=\frac{\varepsilon}{n(n-2)C}$, $C=\max_{ B_1^{+}\setminus B_r^{+}} v>0$ and $r\rightarrow 0.$
Otherwise, it follows from \eqref{lemma-singularity-4-1} that
\begin{equation}\label{lemma-singularity-4-2}
  \left\{
  \aligned
   & \Delta F_3 <0 \ \ \ \text{in} \ \
   B_{1}^{+} \setminus
   {B}_r^{+},   \\
    &\frac {\partial F_3}{ \partial t} <0 \ \ \  \text{on} \ \ \partial
    ({{B}_{1}^{+}} \setminus { B_r^{+}})
    \cap \partial H.
    \endaligned
    \right.
\end{equation}
Now, we will show
    \begin{equation}\label{lemma-singularity-4-3}
    F_3 \ge 0 \ \ \ \text{in}\ \
    \overline{B_1^+}\setminus B_r^{+}.
   \end{equation}
On $ \partial B_r^{+}\cap  \partial B_r$,
 $ F_3 > w_{\lambda_0} > 0$; on $ \partial B_1^{+}
\cap  \partial B_1$, $ F_3\geq  \varepsilon- ( \frac
{\epsilon}{2} - r^{n-2} {\epsilon} + \frac {\epsilon
}2t^2) >  0$. Suppose
that \eqref{lemma-singularity-4-3} fails,  it follows
from the Strong Maximum Principle that  there exists some $ x_0=( x_0', 0)$
with $r<| x_0'|<1$ such that
$$
F_3( x_0)=  \min_{
\overline{B_1^+ } \setminus { B_r^{+}\cap \partial H}} F_3  <0,
$$
which contradicts  to the  boundary condition of \eqref{lemma-singularity-4-2}.
This establishes \eqref{lemma-singularity-4-3}.
Sending $ r \to 0$, we obtain Lemma \ref{singularity-4}.


\smallskip
\noindent{\bf Proof of Lemma \ref{normal-derivative}.}
Without loss of generality, we assume $\lambda_0=1.$ Set $\Omega_2=\{x=(x', t)|x\in {B_1^{+}}\setminus{\overline{ B_{1/2}^{+}}}, t<a\}$ with $0<a\leq 1/2$ to be chosen later.
Due to the continuity of $v$ in $H\setminus \{0\},$ we know that:  there exists some positive constant $C_1$ such that
 \begin{equation}\label{normal-derivative-1}
    v_{\lambda_0}(x)<C_1<+\infty \ \ \ \text{for}\ \
    x\in \overline{\Omega_2}.
   \end{equation}
Let
$$
h_2(x)=\varepsilon ( \frac{1}{|x|^{n-2}}-1)(\frac{t^2 \mu}{2}+ \frac{1-\mu}{(2|x|)^{n-2}} +1-\mu),\ \ \  x\in \Omega_2,
$$
where $0<\varepsilon,\mu<1$ will be chosen later. Let $A_2(x)= w_{\lambda_0}(x)- h_2(x),$ then $A_2(x)$ satisfies the following equation.
\begin{equation}\label{normal-derivative-2}
  \left\{
  \aligned
   \Delta A_2(x) &=c_5(x) w_{\lambda_0}(x)- \Delta h_2(x) \quad\quad \text{in} \ \Omega_2\\
    \frac {\partial A_2}{ \partial t} &=- c_6(x) w_{\lambda_0}(x)\leq 0   \quad\quad  \text{on} \ \
   \partial {\Omega_2}  \cap \partial H
    \endaligned
    \right.
\end{equation}
where $c_5(x)$ and $c_6(x)$ are given in \eqref{kelvin-minus}, and
\begin{equation}\label{normal-derivative-2-1}
\Delta h_2(x)=\epsilon [ \mu(\frac{1}{|x|^{n-2}}-1)+2(n-2)(\frac{(1-\mu)(n-2)}{2^{n-2}|x|^{2n-2}}-\frac{\mu t^2}{|x|^n})].
\end{equation}

For suitably chosen $\varepsilon$ and $\mu,$ we want to show
\begin{equation}\label{normal-derivative-3}
A_2(x)= w_{\lambda_0}(x)- h_2(x)\geq 0, \quad \forall x \in \Omega_2.
\end{equation}
Using \eqref{proposition 3.5-2}, we can choose $\varepsilon_0>0$ small enough, such that for all $0<\varepsilon<\varepsilon_0,$ we have $A_2(x)\geq 0$ on $\partial {\Omega_2} \cap \{\partial B_{1/2}\cup \{ t=a\}\}.$ Also, from the construction of $h_2$ we know $A_2(x)=0$ on $\partial {\Omega_2} \cap \partial B_1.$ Suppose the contrary to \eqref{normal-derivative-3}, then there exists some $\bar{x}=(\bar{x}', \bar{t})\in \overline{\Omega}$ such that
\begin{equation}\label{normal-derivative-4}
A_2(\bar{x})= \min_{\overline{\Omega}_2}A_2< 0.
\end{equation}
From the above and the boundary condition of \eqref{normal-derivative-2}, we have $0<\bar{t}<a,\ 1/2 < |\bar{x}|<1.$ Thus
\begin{equation}\label{normal-derivative-5}
\Delta A_2(\bar{x})\geq 0.
\end{equation}
From \eqref{normal-derivative-4}, we have $v(\bar{x})-v_{\lambda_0}(\bar{x})-h_2(\bar{x})<0,$ and then
\begin{equation}\label{normal-derivative-6}
v(\bar{x})<C_2<+\infty
\end{equation}
for some constant $C_2$ depending only on $C_1.$ By \eqref{normal-derivative-1} and \eqref{normal-derivative-6} we have $c_5(x)<C_3$ for some positive constant $C_3$ depending only on $C_1, C_2$ and $\lambda_0.$ Hence, it follows from \eqref{normal-derivative-2}, \eqref{normal-derivative-2-1} and \eqref{normal-derivative-5} that
\begin{equation}\label{normal-derivative-7}
C_3 w_{\lambda_0}(\bar{x})\geq c_5(\bar{x}) w_{\lambda_0}(\bar{x})>\varepsilon \mu (\frac{1}{|x|^{n-2}}-1),
\end{equation}
for $\bar{t}<a$ if we choose   $0<a \leq \big((1-\mu)\mu^{-1}2^{2-n}(n-2)\big)^{1/2}$ such that $\frac{(1-\mu)(n-2)}{2^{n-2}|x|^{2n-2}}-\frac{\mu t^2}{|x|^n}>0.$
Again by \eqref{normal-derivative-4},
\begin{equation}\label{normal-derivative-8}
  \aligned
  w_{\lambda_0}(\bar{x}) < h_2(\bar{x})& < \varepsilon (\frac{1}{|x|^{n-2}}-1)[\frac{\mu \bar{t}^2}{2}+2(1-\mu)]\\
    & <\varepsilon (\frac{1}{|x|^{n-2}}-1)[\frac{\mu }{2C_3}+2(1-\mu)],
    \endaligned
\end{equation}
for  $\bar{t}<a$ if we choose $0<a \leq {C_3}^{-1/2}.$ Therefore,  if we choose $a\leq \min \{ \big((1-\mu)\mu^{-1}2^{2-n}(n-2)\big)^{1/2},  {C_3}^{-1/2}, 1/2\}$ such that both \eqref{normal-derivative-7} and \eqref{normal-derivative-8} hold, then $$
C_3 \varepsilon (\frac{1}{|x|^{n-2}}-1)[\frac{\mu }{2C_3}+2(1-\mu)]> \varepsilon \mu (\frac{1}{|x|^{n-2}}-1);
$$
that is, $\mu<\frac{4 C_3}{1+4C_3}.$ If $\mu$ is chosen in such a way that $\mu>\frac{4 C_3}{1+4C_3}$ from the beginning, we reach a contradiction. Thus \eqref{normal-derivative-3} holds. Since we also know that $A_2(x_0)=0,$ we have
$$
\frac {\partial A_2}{\partial \nu} (x_0)\geq 0.
$$
It follows from a direct computation that
$$
\frac {\partial w_{\lambda_0}}{\partial \nu}(x_0) =\frac {\partial A_2}{\partial \nu} (x_0)+\frac {\partial h_2}{\partial \nu} (x_0) \geq \frac {\partial h_2}{\partial \nu} (x_0)=\varepsilon (n-2)(2^{2-n}+1)(1-\mu) >0.
$$
Hence, Lemma \ref{normal-derivative} is established.

We hereby complete the proof of the proposition.

\end{proof}

\medskip
\begin{proposition}\label{lambda-stop-cor}
Suppose $\lambda_b>0$ for some $b \in \partial H, $ then we have
$\lambda_b>0$ for all $b \in \partial H. $
\end{proposition}
\begin{proof}
This proposition can be proved in the same way as that for Claim 4 of Proposition 2.1 in \cite{LZ1995}.
\end{proof}

\medskip

 We will also need the second technical Li-Zhu lemma which first appeared in  \cite[Lemma 2.5]{LZ1995}
 \begin{lemma}\label{keylemma2} (Lemma~$2.5$~ in~\cite{LZ1995}) Suppose $f\in C^1(H)\ (n\geq 3)$  satisfying: for all $b\in \partial H=\Bbb{R}^{n-1} ;$
there exists $\mu_b\in \Bbb{R}$ such that
$$
f(x'+b)= \frac {\mu_b ^{n-2}}{|x'|^{n-2}} f\Big( \frac {\mu_b^2 x'} {|x'|^2}+b \Big), \quad \forall x'\in \partial H\setminus \{0\}.
$$
Then for some $a\geq 0,\ d>0,\ x_0'\in \partial H,$
$$f(x')=\pm \big(\frac{a}{|x'-x_0'|^2+d}\big)^{(n-2)/2}, \quad    \forall x'\in\partial H.$$
\end{lemma}

\medskip
It is easy to see from Proposition \ref{lambda-stop}, Proposition \ref{lambda-stop-cor} and Lemma \ref{keylemma2} that
\begin{lemma}\label{boundary-solution} Let $u$ be a positive function satisfying the hypotheses of Theorem \ref{Class-critical}. Suppose $\lambda_b>0$ for some $b \in \partial H= \Bbb{R}^{n-1}, $ then we have for all $a,d>0,\ x_0'\in \partial H,$
$$
u(x',0)= \big(\frac{a}{|x'-x_0'|^2+d}\big)^{(n-2)/2}, \quad    \forall x'\in\partial H.
$$
\end{lemma}

\medskip
{\bf Proof of Theorem \ref{Class-critical}.}\ We divide  the proof of this theorem into two cases.

Case 1.\ Suppose $\lambda_b=0$ for all $b \in \partial H= \Bbb{R}^{n-1}, $ then from Proposition \ref{lambda-notstop} we have
 $
u(x)=u(x', t)=u(0, t)=u(t)$ for $\forall x=(x',t)\in H.
$ Moreover, in terms of value range of  $c_+,$  $u(t)$ has three different expressions as in Proposition \ref{lambda-notstop}.

Case 2.\ Suppose $\lambda_b>0$ for some $b \in \partial H= \Bbb{R}^{n-1}, $ then from Lemma \ref{boundary-solution} we have for all $a,d>0,\ x_0'\in \partial H,$
$$
u(x',0)= \big(\frac{a}{|x'-x_0'|^2+d}\big)^{(n-2)/2}, \quad    \forall x'\in\partial H.
$$
As in \cite{LZ1995}, let $x_0=(x_0', -\sqrt{d}),$ set
$$
\varphi(x)=\frac {1}{|x-x_0|^{n-2}} u\Big( \frac {x-x_0} {|x-x_0|^2}+x_0\Big),
$$
and  $B=\{ (x', t)| [ (t+ \sqrt{d})/(|x'-x_0'|^2+(t+\sqrt{d})^2)]-\sqrt{d}>0\}.$ Actually, $B$ is a ball in $\Bbb R^n.$ Without loss of generality, we assume $a=1$ in Lemma
\ref{keylemma2}. A direct computation as in \cite{LZ1995} shows that $\varphi(x)=1$ on $\partial B.$
Moreover, we have
$$
\Delta \varphi (x)=n(n-2)\varphi ^p\quad \text{in}\ B.
$$
It is easy to see that the solution to this equation is unique. And it follows from the Strong Maximum Principle that $\varphi<1$  in $B.$
Define $\psi(x)=1-\varphi(x).$ Clearly, $\psi$ satisfies
\begin{equation}\label{theorem2.1-1}
  \left\{
  \aligned
   - \Delta \psi &= n(n-2)(1-\psi)^p\ \  \ \ \ \  \ \ \ \text{in} \ B \\
 \psi &=0   \ \ \ \  \ \ \  \ \ \ \   \ \ \  \ \ \ \   \ \ \   \ \ \ \  \ \ \ \  \text{on}  \
   \partial B,
    \endaligned
    \right.
\end{equation}
where $0<\psi<1$ in $B.$ Applying the result of \cite{GNN1979}, we know that $\psi$ is radically symmetric about the center of $B.$ Hence $\varphi(x)$ must take the form
$$
\varphi (x)=\big(\frac{\varepsilon}{\varepsilon ^2-|x-x_1|^2} \big)^{(n-2)/2}
$$
 for some $\varepsilon>0$ and $ \ x_1\in
\Bbb{R}^{n}.$ Then
\begin{equation}\label{theorem2.1-2}
   \aligned
   u(y) &=\frac {1}{|y-x_0|^{n-2}} \varphi\Big( \frac {y-x_0} {|y-x_0|^2}+x_0\Big), \\
 &= \big( \frac{\varepsilon}{(\varepsilon ^2-|x_0-x_1|^2)|y-x_0|^2-2(y-x_0)(x-x_0)-1  } \big)^{(n-2)/2} .
    \endaligned
\end{equation}
We hereby complete the proof of Theorem  \ref{Class-critical}.

 \vskip 1cm
 \noindent {\bf Acknowledgements}\\
 \noindent  S. Tang and L. Wang are supported by the China Scholarship Council for their study/research at the University of Oklahoma.  S. Tang and L. Wang would like to thank Department of Mathematics at  the  University of Oklahoma for its hospitality,  where this work has been done. The work of S. Tang is supported by the
National Natural Science Foundation of China (Grant No. 11801426) and Natural Science Basic Research Plan in Shaanxi Province of China (Program No. 2017JQ1022).

\end{document}